\documentclass[reqno]{amsart}

\usepackage{amsmath,amsthm,amsfonts,amssymb,young}
\usepackage{array, boldline, makecell, booktabs}
\usepackage{bm}
\usepackage{ytableau}
\usepackage{euscript, mathrsfs} 
\usepackage{tikz}
\usetikzlibrary{backgrounds}
\usetikzlibrary{patterns, decorations.pathreplacing}
\usepackage{mathtools}
\usepackage[colorlinks]{hyperref}
\usepackage{cleveref}
\usepackage[margin=1in]{geometry}
\numberwithin{equation}{section}    
\usepackage[all]{xy}
\usepackage{graphicx,import}
\usepackage{amscd}
\usepackage{color}
\usepackage{enumerate}
\usepackage{fourier}
\usepackage{cases}
\usepackage{cite}
\usepackage[T1]{fontenc}

\usepackage{amsthm}

\newtheorem*{maintheorem}{Main Theorem}

\newtheorem{thm}{Theorem}[section]
\newtheorem{lem}[thm]{Lemma}
\newtheorem{proposition}[thm]{Proposition}
\newtheorem{corollary}[thm]{Corollary}
\theoremstyle{definition}

\newtheorem{definition}[thm]{Definition}

\newcommand{\ZZ}{\mathbb{Z}}
\newcommand{\corner}{\mathsf{C}}
\newcommand{\coeff}{\mathsf{c}}
\newcommand{\xchar}[2]{x_{#1}(#2)}
\newcommand{\lchar}[2]{\sh(#1_{\le{#2}})}
\newcommand\Dyck{\operatorname{Dyck}}

\DeclareMathOperator{\wt}{wt}
\DeclareMathOperator{\Par}{Par}
\DeclareMathOperator{\dg}{dg}
\DeclareMathOperator{\inv}{inv}
\DeclareMathOperator{\Inv}{Inv}

\DeclareMathOperator{\sh}{sh}
\DeclareMathOperator{\SYT}{SYT}
\DeclareMathOperator{\Des}{Des}
\DeclareMathOperator{\Add}{Add}
\DeclareMathOperator{\FPI}{FPI}

\title{Haglund--Haiman--Loehr formula via Carlsson--Mellit algebra}

\author{Younggwang Cho}
\address{Department of Mathematics \\ Sungkyunkwan University \\ Suwon \\ Korea}
\email{brglory@g.skku.edu}

\author{Jaeseong Oh}
\address{Department of Mathematics \\ Sungkyunkwan University \\ Suwon \\ Korea}
\email{jaeseongoh@skku.edu}

\begin{document}

\begin{abstract}
We prove a Haglund--Haiman--Loehr type combinatorial formula for the torus fixed point classes $I_{\mu,w}$ in the equivariant $K$-theory $K_{\mathbb{C}^{*}\times\mathbb{C}^{*}}(\operatorname{PFH}_{n,n-k})$ of parabolic flag Hilbert schemes. Under the identification of Bechtloff Weising and Orr, our result gives an HHL type formula for modified partially symmetric Macdonald functions in terms of weighted partial Dyck paths. When $w=\emptyset$, it specializes to the classical HHL formula. The proof relies essentially on the Carlsson--Gorsky--Mellit action of the Carlsson--Mellit algebra $\mathbb{A}_{q,t}$.
\end{abstract}

\maketitle

\section{Introduction}
\label{sec:introduction}
\subsection{Main result}

Macdonald polynomials form one of the central families of symmetric
functions \cite{Mac88}, with connections to algebraic combinatorics, representation
theory, and geometry. Among their several normalizations, the modified
Macdonald polynomials
\(
\widetilde H_\mu(X;q,t)
\)
are particularly well suited to combinatorial and geometric
applications. A central problem in their theory was the
Schur positivity of
\[
\widetilde H_\mu(X;q,t)
=
\sum_{\lambda\vdash|\mu|}
\widetilde K_{\lambda,\mu}(q,t)s_\lambda(X).
\]
Haiman proved that
$\widetilde K_{\lambda,\mu}(q,t)\in\mathbb{Z}_{\geq0}[q,t]$ using the
geometry of the Hilbert scheme of points in the plane
\cite{Hai01}. In the corresponding fixed point picture, the
modified Macdonald polynomial $\widetilde H_\mu$ is associated with the
torus fixed point $I_\mu\in\operatorname{Hilb}^{|\mu|}(\mathbb{C}^2)$.

Although Haiman's theorem establishes Schur positivity, it does not
itself give a manifestly positive combinatorial formula for the Schur
coefficients.  The most prominent explicit combinatorial formula for
$\widetilde H_\mu$ is the Haglund--Haiman--Loehr formula.  If $\alpha$
is any rearrangement of the parts of $\mu'$, then
\begin{equation}\label{eq:intro-hhl}
  \widetilde H_\mu(X;q,t)
  =
  \sum_{\sigma:\dg'(\alpha)\rightarrow\mathbb{Z}_{>0}}
  q^{\inv_\alpha(\sigma)}
  \prod_{u\in\Des_\alpha(\sigma)}
  q^{-a_\alpha(u)}t^{l_\alpha(u)+1}
  x^\sigma,
\end{equation}
where the sum ranges over fillings of the column diagram of $\alpha$
\cite{HHL05,HHL08}.

The HHL formula admits a reformulation in terms of the weighted characteristic functions \(\chi(\pi,\wt)\) introduced by Carlsson and Mellit \cite[Section~3.2]{CM18}.
For each \(\alpha\), the reading order on the cells of \(\dg'(\alpha)\) determines a Dyck path \(\pi_\alpha\): the pairs below \(\pi_\alpha\) encode the attacking pairs, while its distinguished corners correspond to the descents. Assigning to the corner \(c_u\) associated with a descent cell \(u\) the weight
$$ \wt_\alpha(c_u) = q^{-a_\alpha(u)}t^{l_\alpha(u)+1}, $$
one obtains
\begin{equation}\label{eq:intro-hhl-characteristic}
\widetilde H_\mu(X;q,t)=
\chi(\pi_\alpha,\wt_\alpha).
\end{equation}

Carlsson, Gorsky, and Mellit extended the fixed point picture from
Hilbert schemes to the parabolic flag Hilbert schemes
$\operatorname{PFH}_{n,n-k}$ and constructed an action of the
Carlsson--Mellit algebra $\mathbb{A}_{q,t}$ on the equivariant $K$-theory \cite{CGM20}.  The torus fixed
points are indexed by pairs $(\mu,w)$, where $\mu$ is a partition and
$w=(w_1,\ldots,w_k)$ records an ordered horizontal strip of $\mu$
\cite{CGM20}.  Under the resulting isomorphism
\[
\bigoplus_{n\geq k}
K_{\mathbb{C}^*\times \mathbb{C}^*}
\bigl(\operatorname{PFH}_{n,n-k}\bigr)
\simeq
V_k
:=
\Lambda\otimes_{\mathbb{Q}(q,t)}
\mathbb{Q}(q,t)[y_1,\ldots,y_k],
\]
the fixed point classes give a distinguished basis
$\{I_{\mu,w}\}$ of $V_k$.  For $k=0$, this basis specializes to the
modified Macdonald basis:
\[
I_{\mu,\emptyset}
=
(-1)^{|\mu|}T_\mu^{-1}\widetilde H_\mu,
\qquad
T_\mu
=
q^{n(\mu')}t^{n(\mu)}
=
\prod_{u\in\mu}q^{i(u)-1}t^{j(u)-1}.
\]

The connection between \eqref{eq:intro-hhl-characteristic} and the
parabolic fixed point basis is provided by the action of the Carlsson--Mellit algebra.
In the polynomial representation of $\mathbb{A}_{q,t}$, the
characteristic function of an unweighted Dyck path $\pi$ is given by
$d_\pi(1)$, where $d_\pi$ is the word in the raising and lowering
operators $d_+$ and $d_-$ in $\mathbb{A}_{q,t}$ determined by $\pi$.  Weighted characteristic
functions can, in turn, be expanded as linear combinations of
unweighted ones.  Carlsson, Gorsky, and
Mellit give explicit formulas for the actions of $d_+$ and $d_-$ on the
fixed point basis $\{I_{\mu,w}\}$ \cite[Lemma~4.2]{CGM20}.  This common
operator structure suggests asking whether the fixed point
basis elements $I_{\mu,w}$ admit a formula extending
\eqref{eq:intro-hhl-characteristic}.

To each fixed point index $(\mu,w)$, we associate a weighted partial
Dyck path $(\pi_{\mu,w},\wt_{\mu,w})$. Its underlying path is determined
by the rearrangements of columns of $\mu$ with the ordering encoded by $w$, and its corner
weights are defined using the corresponding arm and leg statistics (for precise definitions of $\pi_{\mu,w}$ and $\wt_{\mu,w}$, see Section~\ref{sec:HHL_for_I_mu_w}).

The following is the introductory form of our main theorem.

\begin{maintheorem}[Theorem~\ref{thm:main}]
  \phantomsection \label{thm:intro-main}
  For every fixed point index $(\mu,w)$, we have
  \begin{equation}\label{eq:intro-main}
    (-1)^{|\mu|}T_\mu I_{\mu,w}
    =
    \chi(\pi_{\mu,w},\wt_{\mu,w}).
  \end{equation}
\end{maintheorem}

When $w=\emptyset$, the construction gives the weighted Dyck path
associated with a rearrangement of the columns of $\mu$, and
\eqref{eq:intro-main} reduces to
\eqref{eq:intro-hhl-characteristic}. Thus
our main theorem extends the HHL formula from modified
Macdonald polynomials to the full parabolic fixed point basis.

The proof exploits the $\mathbb{A}_{q,t}$-algebra and its action on the fixed point basis.
In \eqref{eq:decomposition_multi}, we first
rewrite the weighted characteristic function as a linear combination of unweighted ones.
We then expand each unweighted characteristic function in the fixed point basis as in
Lemma~\ref{lem:c_pi_T}. The resulting weighted expansion is indexed by
standard Young tableaux as in \eqref{eq:eq:I_expansion_wt}. Then Lemma~\ref{lem:factorization} provides a product formula for the contribution for each standard Young tableau. For the path
associated with $(\mu,w)$, local vanishing factors force a unique tableau
\( \tau^* \) to survive, and its coefficient telescopes to
$(-1)^{|\mu|}T_\mu$.

This use of $\mathbb{A}_{q,t}$ fits into a broader pattern.  The algebra
was central to Carlsson and Mellit's proof of the Shuffle Conjecture and
to Mellit's proof of its rational version \cite{CM18,Mel21}.  More
recently, following Hikita's proof of the Stanley--Stembridge conjecture
\cite{Hik24}, Griffin, Mellit, Romero, Weigl, and Wen used the
$\mathbb{A}_{q,t}$-action to give another proof and recover Hikita's
formula \cite{GMRWW25}.  Our result adds the HHL formula and its
parabolic extension to this list of applications.

\subsection{Comparison with partially symmetric Macdonald theory}

The fixed point basis $I_{\mu,w}$ is closely related to partially
symmetric Macdonald polynomials.  Goodberry's type $A$ partially
symmetric Macdonald polynomials agree, up to conventions, with the
$m$-symmetric Macdonald polynomials studied by Lapointe and subsequently
by Concha and Lapointe \cite{Goo24,Lap25,CL23}.  Goodberry and Orr
introduced the modified functions
\(
\widetilde H_{(\lambda\mid\gamma)}(X\mid y;q,t)
\)
and conjectured that they realize the normalized fixed point classes of
parabolic flag Hilbert schemes \cite{Goo23}.  This conjecture was proved
by Bechtloff Weising and Orr
\cite[Theorem~6.12]{BWO25}.  Denoting their bijection of indices by
$\phi(\mu,w)=(\lambda\mid\gamma)$ and translating their conventions to
ours, their theorem gives
\begin{equation}\label{eq:intro-partial-macdonald}
(-1)^{|\mu|}T_\mu I_{\mu,w}
=
\widetilde H_{\phi(\mu,w)}(X\mid y;q,t).
\end{equation}
Consequently, our main theorem, equivalently \eqref{eq:intro-main}, gives an HHL type
formula for modified partially symmetric Macdonald functions:
\[
\widetilde H_{\phi(\mu,w)}(X\mid y;q,t)
=
\chi(\pi_{\mu,w},\wt_{\mu,w}).
\]

Blasiak, Haiman, Morse, Pun, and Seelinger obtained a closely related
HHL type formula in terms of flagged LLT polynomials.  More precisely,
they give a signed flagged-LLT expansion for the right-stable integral
forms $J_{\eta\mid\lambda}$
\cite[Proposition~7.5.2]{BHMPS25}, and relate these forms to modified
right-stable functions through a $q$-shifted nonsymmetric plethysm
\cite[Theorem 7.6.6]{BHMPS25}.  In view of
\eqref{eq:intro-partial-macdonald}, their result is closely parallel to
\eqref{eq:intro-main} after translating indices and
normalizations.  We do not pursue a term by term correspondence between
their flagged LLT summands and our weighted partial Dyck path
characteristic function.

The paper is organized as follows. Section~\ref{sec:Preliminaries} reviews the required
background on symmetric functions, weighted characteristic functions,
the Carlsson--Mellit algebra, and the fixed point basis $I_{\mu,w}$.
In Section~\ref{sec:HHL_for_I_mu_w}, we construct the weighted partial Dyck paths for $(\mu,w)$ and state the
main theorem.  Section~\ref{sec:proof} proves the theorem using the
Carlsson--Mellit action.

\section{Preliminaries}\label{sec:Preliminaries}
\subsection{Partitions and compositions} \label{sec:partition}
A \emph{partition} is a weakly decreasing finite sequence \( \lambda = (\lambda_1 ,\dots, \lambda_{\ell}) \) of positive integers.
The integers \( \lambda_1 ,\dots, \lambda_\ell \) are called \emph{parts},
and the number of parts \( \ell = \ell(\lambda) \) is called the \emph{length}.
If \( \sum_{i=1}^\ell \lambda_i = n \),
we denote \( \lambda \vdash n \) or \( |\lambda| = n \).
The set of all partitions is denoted by \( \Par \). The \emph{Young diagram} is one way to illustrate a partition.
For a partition \( \lambda \), a Young diagram is defined by
\[
  \dg(\lambda) = \{(i, j) : 1 \le j \le \ell, 1 \le i \le \lambda_j \}.
\]
Each element in a Young diagram is called a \emph{cell}. A \emph{composition} is a sequence \( \alpha = (\alpha_1 ,\dots, \alpha_\ell) \)
of positive integers of finite length.
If \( \sum_{i=1}^\ell \alpha_i = n \),
we denote \( \alpha \vDash n  \) or \( |\alpha| = n \).

Similar to the Young diagram of a partition,
a \emph{column diagram} of a composition is defined by
\[
  \dg'(\alpha) = \{(i, j) : 1 \le i \le \ell, 1 \le j \le \alpha_i \}.
\]
For a given partition $\lambda$, we say that a cell $u$ is \emph{addable} if $\lambda\cup\{u\}$ is also a partition. We denote the set of addable cells of $\lambda$ by $\Add(\lambda)$.
For a cell \( u = (i, j)\) in a Young diagram \( \dg(\lambda) \)
of a partition \( \lambda \),
the arm length \( a(u) \) is the number of cells
to the right of the cell \( u \),
and the leg length \( l(u) \) is the number of cells
above the cell \( u \), i.e.,
\[
  a(u) = \left| \{(i', j) \in \dg(\lambda) : i' > i\} \right|, \quad
  l(u) = \left| \{(i, j') \in \dg(\lambda) : j' > j \} \right|.
\]

Let $\alpha$ be a composition. For a cell \( u = (i, j) \) in a column diagram \( \dg'(\alpha) \),
the leg length \( l_\alpha(u) \) is defined by
\[
  l_\alpha(u) = \left| \{(i, j') \in \dg'(\alpha) : j' > j \} \right|,
\]
and if \( j > 1 \), the arm length \( a_\alpha(u) \) is defined by
\[
  a_\alpha(u) = \left| \{ (i', j) \in \dg'(\alpha) : i' > i, \alpha_{i'} \le \alpha_i \} \right|
  + \left| \{ (i', j-1) \in \dg'(\alpha) : i' < i, \alpha_{i'} < \alpha_i \} \right|.
\]

We note that our definition of leg length and arm length is different from \cite{HHL08}.
By rearranging columns in reverse order,
we can recover their definition.
We also caution that it is necessary to mention
whether a sequence is a partition or a composition,
because the definitions of arm lengths and leg lengths for a composition
do not generalize those for a partition.

A \emph{standard Young tableau} of a partition $\lambda$
is a bijective function $\tau:\dg(\lambda)\rightarrow\{1,2,\dots,|\lambda|\}$
satisfying \( \tau(i, j) < \tau(i+1, j) \) and
\( \tau(i, j) < \tau(i, j+1) \).
The set of standard Young tableaux of size \( n \)
is denoted by \( \SYT(n) \).
For a standard Young tableau \( \tau \) and an integer \( i \),
we define \( \tau_{\le i} \) to be a standard Young tableau
given by the entries not greater than \( i \).

A \emph{filling} of the column diagram $\dg'(\alpha)$ of the composition \( \alpha \) is a function
\[
  \sigma: \dg'(\alpha) \rightarrow \ZZ_{>0}.
\]

For each cell \( u = (i, j) \),
we define the \emph{character} of the cell by
\( q^{i-1}t^{j-1} \).
The character of a partition \( \lambda \) is the product
of the characters of cells in \( \lambda \) and is denoted by \( T_{\lambda} \).
We will often refer to a cell in the partition \( \lambda \) by its character.

For a column diagram of a composition we define the \emph{reading order}.
We read the cells row by row, from the top to the bottom,
and in each row from the left to the right,
skipping all the blanks.

\begin{figure}
  \centering
  \begin{tikzpicture}[scale = 0.95]
    \begin{scope}
      \draw (0, 0) rectangle +(1, 1);
      \draw (3, 0) rectangle +(1, 1);
      \node at (0.5, 0.5) {\(u_i\)};
      \node at (3.5, 0.5) {\(u_j\)};
      \node at (2, -0.45) {\small same row};
    \end{scope}
    \begin{scope}[shift = {(6, 0)}]
      \draw (0, 0) rectangle +(1, 1);
      \draw (3, 1) rectangle +(1, 1);
      \node at (0.5, 0.5) {\(u_j\)};
      \node at (3.5, 1.5) {\(u_i\)};
      \node at (2, -0.45) {\small consecutive rows};
    \end{scope}
  \end{tikzpicture}
  \caption{Attacking pairs in reading order.}
  \label{fig:attacking_pair}
\end{figure}
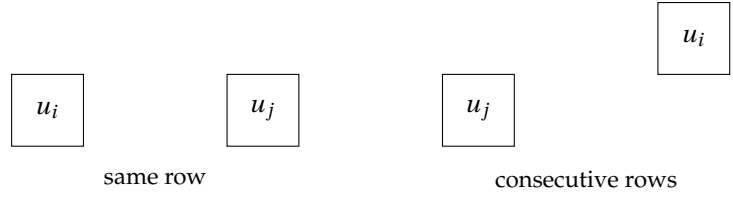

Let \( u_i \) and \( u_j \) be the \(i\)th cell and the \(j\)th cell
in reading order, with \( i < j \).
Then \( (u_i, u_j) \) is said to be an \emph{attacking pair}
if they are either in the same row,
or they are in consecutive rows and \( u_i \) is to the right of \( u_j \).
See \Cref{fig:attacking_pair} for an illustration of attacking pairs.

For a filling \( \sigma \) of a column diagram \( \dg'(\alpha) \), the \emph{inversion set} \( \Inv_\alpha(\sigma) \) is the set of attacking pairs
\( (u, v) \) such that \( \sigma(u) > \sigma(v) \). The \emph{descent set} \( \Des_\alpha(\sigma) \) is the set of cells
whose entry is strictly greater than the entry of the cell directly below.

\subsection{Symmetric functions}
We refer to \cite{Sta01} and \cite{Mac88} for
the detailed definitions of symmetric functions.
Denote the space of symmetric functions over base ring \( \mathbb{Q}(q, t) \) by \( \Lambda \).
A \emph{monomial symmetric function} \( m_{\lambda} \) is defined to be
\[
  m_{\lambda} = \sum_{\alpha} x^\alpha,
\]
where the sum ranges over all (weak) compositions permuting the partition \( \lambda \).
Here and throughout the paper we use the notation
\( x^\alpha = x_1^{\alpha_1}x_2^{\alpha_2} \cdots \).
A \emph{complete homogeneous symmetric function} \( h_{\lambda} \) is defined to be
\begin{align*}
  h_n &= \sum_{\lambda \vdash n} m_{\lambda} = \sum_{i_1 \le \cdots
        \le i_n} x_{i_1} \cdots x_{i_n}, \quad h_0 = 1, \\
  h_{\lambda} &= h_{\lambda_1} \cdots h_{\lambda_\ell}.
\end{align*}
A \emph{power sum symmetric function} \( p_{\lambda} \) is defined to be
\begin{align*}
  p_n &= m_{(n)} = \sum_{i \ge 1} x_i^n, \quad p_0 = 1, \\
  p_{\lambda} &= p_{\lambda_1} \cdots p_{\lambda_\ell}.
\end{align*}
Each of the sets \( \{m_\lambda\}_{\lambda \in \Par} \), \( \{h_\lambda\}_{\lambda \in \Par} \),
and \( \{p_\lambda\}_{\lambda \in \Par} \) forms a basis for \( \Lambda \).

We use brackets to denote \emph{plethystic substitution}. For a symmetric function
\( f \), written as a polynomial in the power sum symmetric function \( p_r \),
\( f[A] \) is obtained by replacing each \( p_r \) with \( p_r[A] \).
Writing \( X = x_1 + x_2 + \cdots \), we have
\[
  p_r[X+Y] = p_r[X] + p_r[Y],
  \qquad
  p_r[qX] = q^r p_r[X].
\]
Thus
\[
  p_r\left[\frac{1}{1-q}X\right]
  = p_r[X+qX+q^2X+\cdots]
  = \sum_{m\geq 0}q^{mr}p_r[X]
  = \frac{p_r[X]}{1-q^r},
\]
which specifies how to compute
\( f\left[\frac{1}{1-q}X\right] \).

The \emph{Macdonald polynomials} \( \{P_{\lambda}(X; q, t)\} \)
are the unique family of symmetric functions determined by the following
\begin{itemize}
\item \( \{P_{\lambda}(X; q, t)\}_{\lambda \in \Par}  \)
  is unitriangular with respect to \( \{m_{\lambda}\}_{\lambda \in \Par} \)
  under the {dominance order}.
\item \( \{P_{\lambda}(X; q, t)\}_{\lambda \in \Par}  \) is an orthogonal basis
  of \( \Lambda \)
  with respect to an inner product defined by
  \[
    \langle p_\lambda, p_\mu \rangle_{q, t}
    = \delta_{\lambda \mu} z_\lambda \prod_{i=1}^\ell \frac{1-q^{\lambda_{i}}}{1 - t^{\lambda_{i}}},
  \]
  where \( z_\lambda = 1^{m_1(\lambda)}m_1(\lambda)!2^{m_2(\lambda)}m_2(\lambda)! \cdots \)
  and  \( m_i(\lambda) \) is the number of \( i \) in the partition \( \lambda \).
\end{itemize}

Since the coefficients of the Macdonald polynomials are rational functions in \( q \) and \( t \),
it is useful to consider scalar multiplication of the Macdonald polynomials,
so that the coefficients are polynomials in \( q \) and \( t \).
The \emph{integral Macdonald polynomial} is defined by
\[
  J_{\lambda}(X; q, t) = \prod_{u \in \lambda} \left(1 - q^{a(u)}t^{l(u)+1}\right)
  P_\lambda(X; q, t).
\]
Finally, the \emph{modified Macdonald polynomial} is defined through plethystic substitution
\[
  \widetilde{H}_{\lambda}(X; q, t) = t^{n(\lambda)}J_{\lambda}\left[\frac{1}{1-t^{-1}}X; q, t^{-1}\right],
\]
where \( n(\lambda) = \sum_{i \ge 1} (i-1)\lambda_i \).

\subsection{Characteristic functions and Carlsson--Mellit algebra}

For given positive integers $k\le n$, a \emph{partial Dyck path} is a lattice path from \( (0,0) \) to \( (k,n) \),
consisting of North and East steps, that does not cross the line \( y=x \).
We denote the set of partial Dyck paths with
\( n \) North steps and \( k \) East steps
by \( \Dyck(n, k) \).
We represent a partial Dyck path by a sequence of
steps \( \pi=(\pi_1,\dots,\pi_{n+k}) \), where each \( \pi_i \) is
either \( + \) or \( - \), denoting a North or East step, respectively.
A partial Dyck path ending at \( (n, n) \) is called a \emph{Dyck path}.
The \emph{corners} of a partial Dyck path \( \pi \) are the cells above \( \pi \)
whose southern and eastern neighbors lie below the path. We denote the
set of corners of \( \pi \) by \( \corner(\pi) \).

For a Dyck path \( \pi \in \Dyck(n, n) \),
the \emph{characteristic function} is defined as follows:
\[
  \chi(\pi) = \sum_{w \in \ZZ^n_{>0}} q^{\inv_\pi(w)} x^w,
\]
where \[ \inv_\pi(w) = \left| \{i < j: w_i > w_j, (i, j) \mbox{ is below } \pi\} \right|. \]
We remark that the definition of the characteristic function \( \chi(\pi) \)
is exactly the same as the definition of a \emph{unicellular LLT polynomial}
corresponding to the Dyck path \( \pi \).

We define a \emph{weighted (partial) Dyck path} as a pair \( (\pi, \wt ) \),
where \( \pi \) is a (partial) Dyck path and
a \emph{weight} \( \wt: \corner(\pi) \rightarrow R \)
is a function for some ring \( R \).
For this paper we only consider weights whose codomain is \( \mathbb{Q}(q, t) \).
We draw a weighted path by drawing a Dyck path \( \pi \)
and writing the weight at the corners.
See \Cref{fig:rearranging} for an example.
We simply skip the weight if the weight is \( 1 \) at the corner.

For \( \pi \in \Dyck(n, n) \) and the weight \( \wt \),
the \emph{weighted characteristic function} of \( (\pi, \wt ) \) is
\[
  \chi(\pi, \wt) = \sum_{w \in \ZZ^n_{>0}}
  q^{\inv_\pi(w)} \left( \prod_{\substack{(i, j) \in \corner(\pi) \\ w_i > w_j}} \wt(i, j) \right) x^w.
\]
Here, we use a different convention from \cite{CM18}.

It is always possible to write a weighted characteristic function
as a linear combination of unweighted characteristic functions.
Let \( \pi \) be a Dyck path
and \( \wt \) be a weight function defined on \( \corner(\pi) \).
For a cell \( c \in \corner(\pi) \), we define
\( \pi^c \) as the Dyck path which goes above the cell \( c \)
and all other steps agree with those of \( \pi \).
Let \( \wt_1 \) be the weight function on \( \pi \) obtained from \( \wt \) by
setting \( \wt_1(c) = 1 \),
and let \( \wt_2 \) be the weight function on \( \pi^c \) that coincides with \( \wt \)
at every corner common to \( \pi \) and \( \pi^c \) and assigns weight \( 1 \) to every other corner.
Then we have a linear relation
\[
  \chi(\pi, \wt) = \frac{q - \wt(c)}{q-1}\chi(\pi, \wt_1) + \frac{-1+\wt(c)}{q-1}\chi(\pi^c, \wt_2).
\]

Generally, suppose that the Dyck path \( \pi \) has \( k \) corners,
\( \corner(\pi) = \{c_1 ,\dots, c_k\} \).
For each subset \( S \subseteq \corner(\pi) \),
let us define \( \pi^S \) to be a path
which goes above the cell \( c_i \) if \( c_i \in S \)
and goes below the cell \( c_i \) if \( c_i \notin S \) and
the remaining steps, which do not involve corners,
are identical to those of \( \pi \).
Let us define the coefficient \( F_S(\wt) \) by
\[
  F_{S}(\wt) = \prod_{c \in S} \frac{-1 + \wt(c)}{q-1}
  \prod_{c \in \corner(\pi) \setminus S} \frac{q - \wt(c)}{q-1}.
\]
Then using these notations we can write the characteristic function
\( \chi(\pi, \wt) \) as follows.
\begin{equation}\label{eq:decomposition_multi}
  \chi(\pi, \wt) = \sum_{S \subseteq \corner(\pi)} F_S(\wt)\chi(\pi^S).
\end{equation}

Carlsson and Mellit defined the algebra
\( \mathbb{A}_{q,t} \) as the quotient of the free algebra
generated by \( d_+,d_+^*,d_-,y_i,z_i,T_i^\pm \) by the relations stated
in \cite[Definition~7.1]{CM18}.
Let \( V_{\bullet} = V_{0} \oplus V_{1} \oplus \cdots \) be the graded vector space,
where \( V_k = \Lambda \otimes_{\mathbb{Q}(q, t)} \mathbb{Q}(q, t)[y_1 ,\dots, y_k] \)
and \( V_0 = \Lambda \).
We describe the action of \(\mathbb A_{q,t}\) on \( V_\bullet \).
The following equations define the actions of \(T_i\),
\(d_+\), \(d_-\), and \(y_i\).

For \( F \in V_k \) and \( 1 \le i < k \), the action of \( T_i \)
on \( V_{\bullet} \) is defined by
\[
  T_i F = \frac{(q-1)y_i F(y_i, y_{i+1}) + (y_{i+1} - qy_i)F(y_{i+1}, y_i)}{y_{i+1}-y_i}.
\]
For \( F \in V_k \) the action of \( d_+ \) is defined by
\[
  d_+F[X] = T_1 \cdots T_k \left( F[X + (q-1)y_{k+1}] \right).
\]
For \( F \in V_{k+1} \), the action of \( d_- \) is defined by
\[
  d_-F[X] = \left (-F[X - (q-1)y_{k+1}] Exp[-y_{k+1}^{-1}X] \right) \Big|_{y_{k+1}^{-1}},
\]
where \( Exp[-y_{k+1}^{-1}X] = \sum_{n \ge 0} h_n[-y_{k+1}^{-1}X] \)
and \( F\Big|_{y_{k+1}^{-1}} \) denotes taking the coefficient of \( y_{k+1}^{-1} \) in \( F \).
The generator \( y_i \) acts on \( V_k \) by multiplication by \( y_i \)
for \( 1 \le i \le k \).
We do not need the actions of the remaining generators here;
see \cite{CM18} for their definitions.

\begin{thm}\cite[Theorem~4.4]{CM18}
  For \( \pi \in \Dyck(n, n) \), let $d_\pi=d_{\pi_{2n}} \cdots d_{\pi_1} $. Then we have
  \[
    \chi(\pi) = d_{{\pi}}(1).
  \]
\end{thm}

Motivated by these results, 
we extend the definition as follows. 
\begin{definition}
  For $\pi \in \Dyck(n, k)$,
  and a weight function $\wt:\corner(\pi)\rightarrow \mathbb{Q}(q,t)$,
  we define the operator $d_\pi$ by
  \[
    d_\pi= d_{\pi_{n+k}} \cdots d_{\pi_1}
  \]
  and the (weighted) characteristic function by
  \[
    \chi(\pi) = d_{\pi}(1), \qquad
    \chi(\pi, \wt) = \sum_{S \subseteq \corner(\pi)} F_S(\wt)\chi(\pi^S).
  \]
\end{definition}

\subsection{The I basis}
Define
\[
  U_k = \bigoplus_{n\ge k}
  K_{\mathbb{C}^{*}\times\mathbb{C}^{*}}(\operatorname{PFH}_{n,n-k}),
  \qquad U_\bullet=\bigoplus_{k\ge 0}U_k.
\]
Thus, \( U_k \) is the direct sum of the equivariant \(K\)-theories
of the parabolic flag Hilbert schemes \(\operatorname{PFH}_{n,n-k}\).
Carlsson, Gorsky, and Mellit \cite{CGM20} defined an action of
\( \mathbb{A}_{q,t} \) on \( U_\bullet \) and proved that
\( U_\bullet \) is isomorphic to \( V_\bullet \) as
\( \mathbb{A}_{q,t} \)-modules.

Let us define the set $\FPI$ as
\[
\bigsqcup_{k \ge 0} \{ (\lambda, w) :
  \lambda \in \Par,
  w = (w_1, \dots, w_k) \text{ forms a horizontal strip of } \lambda,
  \lambda \setminus \{w_1, \ldots, w_i\} \in \Par
  \text{ for all } 1 \le i \le k\}.
\]
The space \( U_\bullet \) has a basis \( \{I_{\lambda,w}\} \)
indexed by \( (\lambda,w) \in \FPI \), with each basis element given
by the class of a torus fixed point. Under the isomorphism
\( U_\bullet\cong V_\bullet \), this basis corresponds to a basis of
\( V_\bullet \), which we denote by the same symbols
\( \{I_{\lambda,w}\} \).

We describe how \( d_+ \) and \( d_- \) act on \( I_{\lambda, w} \).
For a partition \( \lambda \) and a cell \( x \in \Add(\lambda) \),
\( \mathcal{R}_{\lambda, x} \) is the set of cells in \( \lambda \)
that are in the same row as \( x \) and \( \mathcal{C}_{\lambda, x} \)
is the set of cells in \( \lambda \) that are in the same column as \( x \).
Then we define
\[
  d_{\lambda, x} = \prod_{u \in \mathcal{R}_{\lambda, x}}
  \frac{q^{a(u)} - t^{l(u) + 1}}{q^{a(u) + 1} - t^{l(u) + 1}}
  \prod_{u \in \mathcal{C}_{\lambda, x}}
  \frac{q^{a(u)+1} - t^{l(u)}}{q^{a(u) + 1} - t^{l(u) + 1}}.
\]
We remark that \( d_{\lambda, x} \) is the coefficient of \( \widetilde{H}_{\lambda + x}(X; q, t) \)
in the modified Macdonald expansion of \( e_1(X)\widetilde{H}_{\lambda}(X; q, t) \)
\cite{GT96}.

The following proposition describes how \( d_+ \) and \( d_- \) act on
\( I_{\lambda, w} \).

\begin{proposition} \cite[Lemma~4.2,~Example~4.3]{CGM20} \label{prop:I_basis}
  We have
  \begin{align} 
    d_+ I_{\lambda, w}
    &= {-q^k}\sum_{x \in \Add(\lambda)} x d_{\lambda, x}
      \left( \prod_{i=1}^k\frac{x-tw_i}{x-qtw_i} \right) I_{\lambda + x, \{x\} \cup w},
      \label{eq:I_basis_plus} \\
    d_- I_{\lambda, w \cup \{x\}}
    &= I_{\lambda, w}, \label{eq:I_basis_minus} \\
    I_{\lambda, \emptyset}
    &=  \frac{(-1)^{|\lambda|}}{T_\lambda} \widetilde{H}_\lambda(X; q, t). \label{eq:I_basis0}
  \end{align}
  Here, for a sequence \( w = (w_1 ,\dots, w_k) \),
  \( \{x\} \cup w \) denotes \( (x, w_1 ,\dots, w_k) \) and
  \( w \cup \{x\} \) denotes \( (w_1 ,\dots, w_k, x) \).
\end{proposition}

\section{Haglund--Haiman--Loehr formula for
  \texorpdfstring{$I_{\mu,w}$}{I(mu,w)}}
\label{sec:HHL_for_I_mu_w}
To state our main theorem, we briefly review the paper \cite{HHL08}.
\subsection{Haglund--Haiman--Loehr formula} \label{sec:HHL} 
For a composition $\alpha$, define the function \( D_\alpha\) as follows:
\begin{equation} \label{eq:HHL_D}
  D_\alpha(X; q, t)
  = \sum_{\sigma:\dg'(\alpha)\rightarrow\mathbb{Z}_{>0}}
  q^{\inv_\alpha(\sigma)}
  \prod_{u\in\Des_\alpha(\sigma)}
  q^{-a_\alpha(u)}t^{l_\alpha(u)+1}
  x^\sigma,
\end{equation}
where \( \inv_\alpha(\sigma) = |\Inv_\alpha(\sigma)| \).

As discussed in the Introduction, we can express \eqref{eq:HHL_D}
in terms of a characteristic function.
For \( \alpha \vDash n \), define \( \pi_\alpha \in \Dyck(n, n) \)
so that a cell \( (i,j) \) lies below \( \pi_\alpha \) if and only if the
\(i\)th and \(j\)th cells of \( \alpha \), in reading order, form an
attacking pair. By construction, \( (i,j) \) is a corner of \( \pi_\alpha \)
if and only if the \(i\)th cell in reading order of \( \alpha \)
lies directly above the \(j\)th cell in reading order.

With this path in place, we regard a filling
\( \sigma: \dg'(\alpha) \rightarrow \ZZ_{>0} \) as a word
\( w \in \ZZ^n_{>0} \) by reading its entries in reading order.
The construction of \( \pi_\alpha \) encodes the statistics of the filling:
the descent set \( \Des_\alpha(\sigma) \) corresponds to the inversions
\(w_i>w_j\) for \( (i,j)\in\corner(\pi_\alpha) \), while the inversion set
\( \Inv_\alpha(\sigma) \) corresponds to the inversions \(w_i>w_j\) for cells
\( (i,j) \) lying below \( \pi_\alpha \). For a cell \(u\in\dg'(\alpha)\) corresponds to a corner $(i,j)\in\corner(\pi_\alpha)$, define
\( \wt_\alpha(i,j)=q^{-a_\alpha(u)}t^{1+l_\alpha(u)} \). Then
\eqref{eq:HHL_D} is equivalent to
\begin{equation} \label{eq:HHL_char_fn}
  D_\alpha(X; q, t)
  = \sum_{w \in \ZZ^n_{>0}} q^{\inv_{\pi_\alpha}(w)}
  \left( \prod_{\substack{(i, j) \in \corner(\pi_\alpha) \\ w_i > w_j}} \wt_\alpha(i, j) \right)x^w
  = \chi(\pi_\alpha, \wt_\alpha).
\end{equation}

Now we state the equivalent statement of the Haglund--Haiman--Loehr formula.
\begin{thm}\cite[Theorem~5.1.1]{HHL08} \label{thm:HHL}
  We have
  \[
    \widetilde{H}_\mu(X; q, t)
    = \chi(\pi_\alpha, \wt_\alpha),
  \]
  where \( \alpha \) is any rearrangement of the parts of the partition \( \mu' \).
\end{thm}

\subsection{Statement of main theorem}\label{subsec:statement_of_main_theorem}

In this subsection, we state our main theorem,
which is the \( I_{\mu, w} \) analogue of \Cref{thm:HHL}.
To do so, we define a weighted partial Dyck path corresponding to
a pair \( (\mu, w) \).

First we define the partial Dyck path \( \pi_{\mu, w} \) and
the weight \( \wt_{\mu, w} \) when \( w \) is maximal, that is, \( |w| = \mu_1 \).
We first reorder the columns of
\( \dg(\mu) \) so that the left-to-right order of the uppermost cells
agrees with the reverse order of the cells in \(w\). The resulting diagram
is a column diagram \( \dg'(\alpha) \) for some composition \( \alpha \), which
we denote by \( \alpha(\mu,w) \). Let \( \pi_{\alpha(\mu,w)} \) be the Dyck path associated to the composition $\alpha(\mu,w)$ constructed as in Section~\ref{sec:HHL}. 
We then define the partial Dyck path \( \pi_{\mu,w} \) by removing the last
\(|w|\) East steps from \( \pi_{\alpha(\mu,w)} \). Similarly, we set
\( \wt_{\mu,w}=\wt_{\alpha(\mu,w)} \).

If \(w\) is not maximal, that is, if \( |w|<\mu_1 \), we can extend \( w \)
to a maximal horizontal strip \( \bar{w} \)
by adding \( \mu_1-|w| \) cells to the end of \(w\).
The weighted partial Dyck path \( (\pi_{\mu,w},\wt_{\mu,w}) \) is then obtained by
appending \( \mu_1-|w| \) East steps to the end of
\( (\pi_{\mu,\bar{w}},\wt_{\mu,\bar{w}}) \), retaining all corner weights.
The extension \( \bar{w} \) may be chosen arbitrarily,
provided that \( (\mu, \bar{w}) \in \FPI \).

Now we state the main theorem.
\begin{thm} \label{thm:main}
  Let \( (\pi_{\mu, w}, \wt_{\mu, w}) \) be the weighted partial Dyck path
  corresponding to \( (\mu, w) \).
  Then we have
  \[
    I_{\mu, w} = \frac{(-1)^{|\mu|}}{T_\mu}\chi( \pi_{\mu, w}, \wt_{\mu, w}).
  \]
  In particular, \Cref{thm:HHL} follows.
\end{thm}

\begin{figure}
  \centering
  \begin{tikzpicture} [scale = 0.8]
    \begin{scope}
      \draw (0, 0) rectangle +(1, 1);
      \draw (0, 1) rectangle +(1, 1);
      \draw (0, 2) rectangle +(1, 1);
      \draw (1, 0) rectangle +(1, 1);
      \draw (1, 1) rectangle +(1, 1);
      \draw (2, 0) rectangle +(1, 1);
      \draw (2, 1) rectangle +(1, 1);
      \draw (3, 0) rectangle +(1, 1);
      \draw (3, 1) rectangle +(1, 1);
      \draw (4, 0) rectangle +(1, 1);
      \draw (5, 0) rectangle +(1, 1);
      \node at (3, -0.5) {\( (\mu, w) \)};
    \end{scope}
    \begin{scope} [shift = {(0.5, 0.5)}]
      \node at (0, 2) {\textcolor{gray}{\( 4 \)}};
      \node at (1, 1) {\textcolor{gray}{\( 6 \)}};
      \node at (2, 1) {\textcolor{gray}{\( 5 \)}};
      \node at (3, 1) {\( 2 \)};
      \node at (4, 0) {\( 3 \)};
      \node at (5, 0) {\( 1 \)};
    \end{scope}
    \begin{scope} [shift = {(12, 0)}]
      \draw (-3, 0) rectangle +(1, 1);
      \draw (-3, 1) rectangle +(1, 1);
      \draw (-3, 2) rectangle +(1, 1);
      \draw (-5, 0) rectangle +(1, 1);
      \draw (-5, 1) rectangle +(1, 1);
      \draw (-4, 0) rectangle +(1, 1);
      \draw (-4, 1) rectangle +(1, 1);
      \draw (-1, 0) rectangle +(1, 1);
      \draw (-1, 1) rectangle +(1, 1);
      \draw (-2, 0) rectangle +(1, 1);
      \draw (-0, 0) rectangle +(1, 1);
      \node at (-2, -0.5) {\( (\wt_\alpha) \)};
    \end{scope}
    \begin{scope} [shift = {(12.5, 0.5)}]
      \node at (-1, 1) {\footnotesize \( q^{-1}t \)};
      \node at (-3, 1) {\footnotesize \( q^{-3}t^2 \)};
      \node at (-3, 2) {\footnotesize \( q^{-2}t \)};
      \node at (-4, 1) {\footnotesize \( q^{-1}t \)};
      \node at (-5, 1) {\footnotesize \( q^{-2}t \)};
      \node at (-0, -2) {};
    \end{scope}
    \begin{scope} [shift = {(19, 0)}]
      \draw (-3, 0) rectangle +(1, 1);
      \draw (-3, 1) rectangle +(1, 1);
      \draw (-3, 2) rectangle +(1, 1);
      \draw (-5, 0) rectangle +(1, 1);
      \draw (-5, 1) rectangle +(1, 1);
      \draw (-4, 0) rectangle +(1, 1);
      \draw (-4, 1) rectangle +(1, 1);
      \draw (-1, 0) rectangle +(1, 1);
      \draw (-1, 1) rectangle +(1, 1);
      \draw (-2, 0) rectangle +(1, 1);
      \draw (-0, 0) rectangle +(1, 1);
      \node at (-2, -0.5) {Reading order of \( \mu \)};
    \end{scope}
    \begin{scope} [shift = {(19.5, 0.5)}]
      \node at (-3, 2) {\( 1 \)};
      \node at (-5, 1) {\( 2 \)};
      \node at (-4, 1) {\( 3 \)};
      \node at (-3, 1) {\( 4 \)};
      \node at (-1, 1) {\( 5 \)};
      \node at (-5, 0) {\( 6 \)};
      \node at (-4, 0) {\( 7 \)};
      \node at (-3, 0) {\( 8 \)};
      \node at (-2, 0) {\( 9 \)};
      \node at (-1, 0) {\( 10 \)};
      \node at (-0, 0) {\( 11 \)};
    \end{scope}
    \begin{scope} [shift = {(7, -5)}]
      \draw (0, 0) rectangle +(1, 1);
      \draw (0, 1) rectangle +(1, 1);
      \draw (0, 2) rectangle +(1, 1);
      \draw (1, 0) rectangle +(1, 1);
      \draw (1, 1) rectangle +(1, 1);
      \draw (2, 0) rectangle +(1, 1);
      \draw (2, 1) rectangle +(1, 1);
      \draw (3, 0) rectangle +(1, 1);
      \draw (3, 1) rectangle +(1, 1);
      \draw (4, 0) rectangle +(1, 1);
      \draw (5, 0) rectangle +(1, 1);
      \node at (3, -0.5) {\( \tau^*(\mu, w) \)};
      \node at (3, -1.5) {};
    \end{scope}
    \begin{scope} [shift = {(7.5, -4.5)}]
      \node at (0, 0) {\( 1 \)};
      \node at (1, 0) {\( 2 \)};
      \node at (2, 0) {\( 3 \)};
      \node at (0, 1) {\( 4 \)};
      \node at (3, 0) {\( 5 \)};
      \node at (1, 1) {\( 6 \)};
      \node at (2, 1) {\( 7 \)};
      \node at (0, 2) {\( 8 \)};
      \node at (4, 0) {\( 9 \)};
      \node at (3, 1) {\( 10 \)};
      \node at (5, 0) {\( 11 \)};
    \end{scope}
  \end{tikzpicture}
  \begin{tikzpicture}
    \begin{scope} [shift = {(0, -3)}, scale = 0.8]
      \foreach \i in {0 ,..., 11}
      {
        \draw[lightgray] (0, \i) -- (11, \i);
        \draw[lightgray] (\i, 0) -- (\i, 11);
      }
      \foreach \i in {1 ,..., 11}
      {
        \node at (-0.5 + \i, -0.5 + \i) {\( \i \)};
      }
      \draw[red, thick] (0, 0) -- (0, 3) -- (1, 3) -- (1, 5) -- (2, 5) -- (2, 6)
      -- (3, 6) -- (3, 7) -- (4, 7) -- (4, 9) -- (5, 9) -- (5, 11) -- (8, 11);

      \node[font=\footnotesize] at (4.5, 9.5) {\( q^{-1}t \)};
      \node[font=\footnotesize] at (3.5, 7.5) {\( q^{-3}t^2 \)};
      \node[font=\footnotesize] at (2.5, 6.5) {\( q^{-1}t \)};
      \node[font=\footnotesize] at (1.5, 5.5) {\( q^{-2}t \)};
      \node[font=\footnotesize] at (0.5, 3.5) {\( q^{-2}t \)};
    \end{scope}
  \end{tikzpicture}
  \caption{An example of constructing \( \pi, \wt \) for the case \( \mu = (6, 4, 1) \) and \( w = (q^5, q^3t, q^4) \).}
  \label{fig:rearranging}
\end{figure}
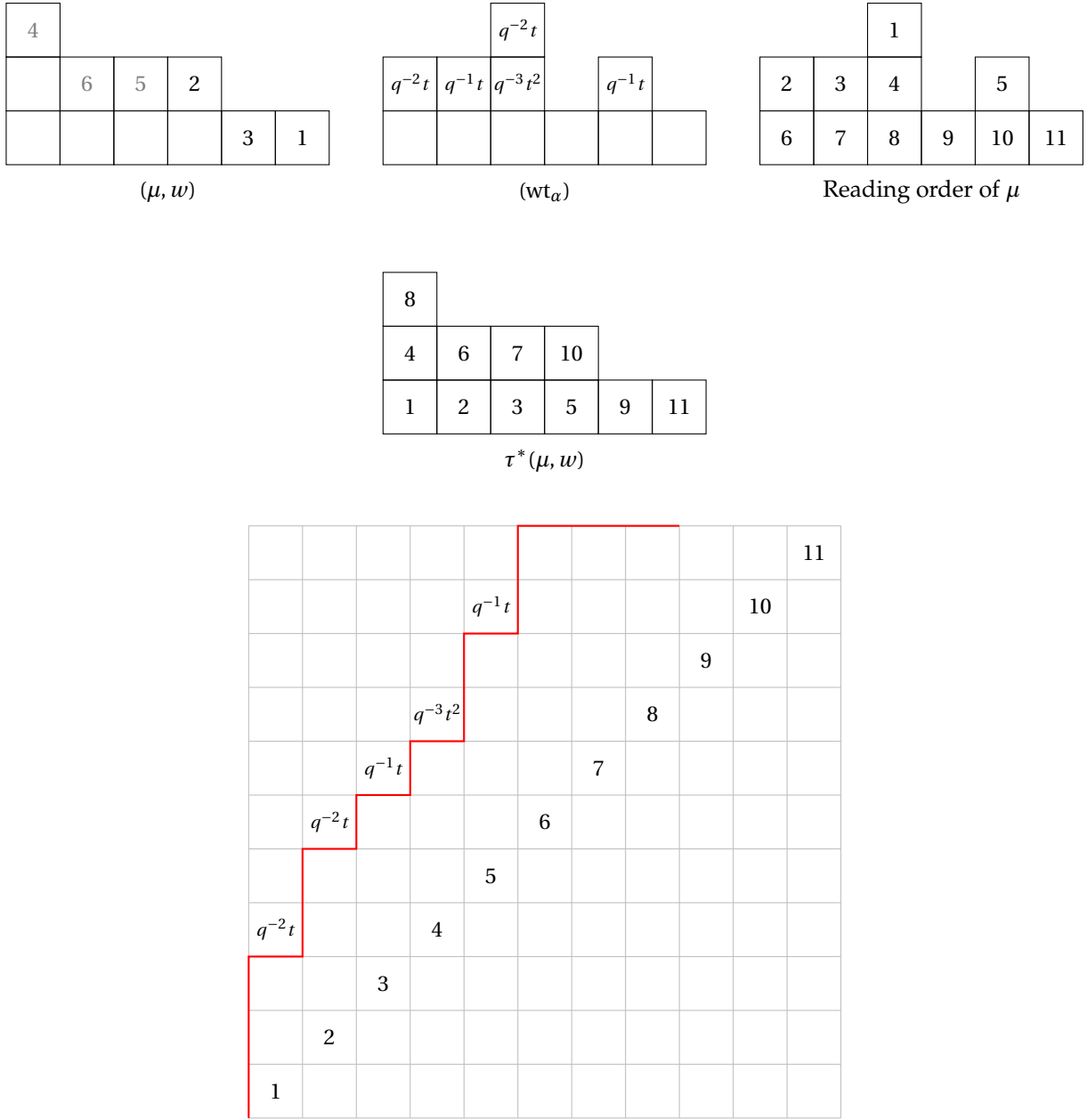

\section{Main proof}\label{sec:proof}
\subsection{SYT decomposition}
For a given fixed point index $(\mu,w)$,
let $(\pi_{\mu,w},\wt_{\mu,w})$ be the weighted partial Dyck path
defined in Section~\ref{subsec:statement_of_main_theorem}.
We define \( \coeff_{(\mu, w), (\lambda, v)} \) as the coefficients of
\( I_{\lambda, v} \) in the expansion of \( \chi(\pi_{\mu, w}, \wt_{\mu, w}) \),
that is,
\begin{equation}\label{eq:I_expansion}
  \chi(\pi_{\mu,w}, \wt_{\mu,w}) = \sum_{(\lambda, v)} \coeff_{(\mu,w),(\lambda,v)}I_{\lambda, v}.
\end{equation}

To prove Theorem~\ref{thm:main}, it suffices to show that
\begin{equation} \label{eq:claim}
  \coeff_{(\mu,w),(\lambda,v)}=
  \begin{cases}
    (-1)^{|\mu|}T_\mu &\text{if} \quad(\mu,w)=(\lambda,v)\\
    0&\text{otherwise}
  \end{cases}    
\end{equation}
To this end, we first refine the coefficients. 

By Proposition~\ref{prop:I_basis},
each time the operator \(d_+\) is applied to \( I_{\lambda, w} \),
the size of the index \( \lambda \) increases by exactly one cell.
Hence, the growth of the partition can be encoded by a standard Young tableau.
For a standard Young tableau \( \tau \),
let \( \xchar{\tau}{i} \) denote the character of
the cell of \( \tau \) containing \( i \).

For \( \tau \in \SYT(n) \), define \( d_\tau \) by
\[
  d_\tau = \prod_{i=1}^n d_{\lchar{\tau}{i-1}, \xchar{\tau}{i}}.
\]
For \( \tau \in \SYT(n) \) and \( \pi \in \Dyck(n, k) \),
we define \( \coeff_{\pi, \tau} \) by
\begin{equation} \label{eq:c_pi_T}
  \coeff_{\pi, \tau} = (-1)^n T_{\sh(\tau)} d_\tau \prod_{j = 1}^{n} \prod_{p=m_j(\pi)+1}^{j-1}
  \frac{q\xchar{\tau}{j} - qt\xchar{\tau}{p}}{\xchar{\tau}{j} - qt\xchar{\tau}{p}},
\end{equation}
where \( m_j(\pi) \) is the number of East steps
before the \( j \)th North step in \( \pi \).

\begin{lem} \label{lem:c_pi_T}
  Let \( \pi \in \Dyck(n, k) \). Then we have
  \[
    \chi(\pi) = \sum_{\tau \in \SYT(n)}\coeff_{\pi, \tau}
    I_{\sh(\tau), (\xchar{\tau}{n} ,\dots, \xchar{\tau}{k+1})}.
  \]
\end{lem}
\begin{proof}
  Let the right hand side be \( \chi'(\pi) \). We argue by induction on \( n+k \).
  If \( n+k=0 \), then \( \pi \) is empty and
  \( \chi(\pi)=d_{\pi}(1)=1=\chi'(\pi) \).
  Now assume \( n+k>0 \), and let
  \( \hat{\pi}=(\pi_1,\dots,\pi_{n+k-1}) \) be the path obtained from \( \pi \) by deleting its last step.

  First suppose that the last step of \( \pi \) is an East step.
  Then \( \hat{\pi} \) and \( \pi \) have the same values of \( m_j(\,\cdot\,) \), so
  \( \coeff_{\hat{\pi},\tau}=\coeff_{\pi,\tau} \) for every \( \tau\in \SYT(n) \).
  By the induction hypothesis,
  \[
    \chi(\hat{\pi})
    =\sum_{\tau\in \SYT(n)} \coeff_{\pi,\tau}\,
      I_{\sh(\tau),(\xchar{\tau}{n},\dots,\xchar{\tau}{k})}.
  \]
  Applying \( d_- \) and using \eqref{eq:I_basis_minus} gives
  \[
    \chi(\pi)=d_-\chi(\hat{\pi})
    =\sum_{\tau\in \SYT(n)} \coeff_{\pi,\tau}\,
      I_{\sh(\tau),(\xchar{\tau}{n},\dots,\xchar{\tau}{k+1})}
    =\chi'(\pi).
  \]

  Next suppose that the last step of \( \pi \) is a North step.
  For each \( \tau\in \SYT(n) \), let \( \hat{\tau} \) be the tableau obtained from \( \tau \) by removing the cell containing \( n \). Then
  \[
    \coeff_{\pi,\tau}
    =-x_\tau(n)\, d_{\lchar{\tau}{n-1},x_\tau(n)}
      \left(\prod_{p=k+1}^{n-1}\frac{q x_\tau(n)-qt\xchar{{\tau}}{p}}{x_\tau(n)-qt\xchar{{\tau}}{p}}\right)
      \coeff_{\hat{\pi},\hat{\tau}},
  \]
  and so
\begin{align*}
        \chi'(\pi)&=\sum_{\tau\in \SYT(n)} -x_\tau(n)\, d_{\lchar{\tau}{n-1},x_\tau(n)}
      \left(\prod_{p=k+1}^{n-1}\frac{q x_\tau(n)-qt\xchar{{\tau}}{p}}{x_\tau(n)-qt\xchar{{\tau}}{p}}\right)
      \coeff_{\hat{\pi},\hat{\tau}}
      I_{\sh(\tau),(\xchar{\tau}{n},\xchar{\tau}{n-1},\dots,\xchar{\tau}{k+1})}
      \\
      &=\sum_{\hat{\tau}\in \SYT(n-1)}\sum_{x \in \Add(\sh(\hat{\tau}))} -x \, d_{\sh(\hat{\tau}),x}
      \left(\prod_{p=k+1}^{n-1}\frac{q x-qt\xchar{\hat{\tau}}{p}}{x-qt\xchar{\hat{\tau}}{p}}\right)
      \coeff_{\hat{\pi},\hat{\tau}}
      I_{\sh(\hat{\tau})\cup\{x\},(x,\xchar{\hat{\tau}}{n-1},\dots,\xchar{\hat{\tau}}{k+1})}.
\end{align*}
By applying \eqref{eq:I_basis_plus}, this equals
  \[
    \sum_{\hat{\tau}\in \SYT(n-1)} \coeff_{\hat{\pi},\hat{\tau}}\,
      d_+ I_{\sh(\hat{\tau}),(\xchar{\hat{\tau}}{n-1},\dots,\xchar{\hat{\tau}}{k+1})} = d_+ \chi'(\hat{\pi})=d_+ \chi(\hat{\pi}).
  \]
Here, the last equation follows from the induction hypothesis.
\end{proof}

We now also refine \( \chi(\pi, \wt) \) as follows.
Recall \eqref{eq:decomposition_multi}:
\[
  \chi(\pi, \wt) = \sum_{S \subseteq \corner(\pi)} F_S(\wt)\chi(\pi^S).
\]
For \( \tau \in \SYT(n), \pi \in \Dyck(n, k) \) and
the weight \( \wt: \corner(\pi) \rightarrow \mathbb{Q}(q, t) \),
we define 
\[
  \coeff_{\pi, \wt, \tau} = \sum_{S \subseteq \corner(\pi)} F_S(\wt) \coeff_{\pi^S, \tau}.
\]

Then by Lemma~\ref{lem:c_pi_T},
we can obtain the refinement of \eqref{eq:I_expansion} immediately.
\begin{equation}\label{eq:eq:I_expansion_wt}
    \chi(\pi, \wt) = \sum_{\tau \in \SYT(n)}
    \coeff_{\pi, \wt, \tau} I_{\sh(\tau), (\xchar{\tau}{n} , \dots, \xchar{\tau}{k+1})}.
\end{equation}

\begin{lem} \label{lem:factorization}
  Let \( \tau \in \SYT(n), \pi \in \Dyck(n, k) \) and
  \( \wt: \corner(\pi) \rightarrow \mathbb{Q}(q, t) \).
  The coefficient \( \coeff_{\pi, \wt, \tau} \) has a factorization
  \[
    (-1)^n T_{\sh(\tau)}d_\tau \prod_{j=1}^n\prod_{p=m_j(\pi)+1}^{j-1} \frac{q\xchar{\tau}{j}-qt\xchar{\tau}{p}}{\xchar{\tau}{j}-qt\xchar{\tau}{p}}
    \prod_{(i, j) \in \corner(\pi)} \left(
      \frac{q-\wt(i, j)}{q-1} - \frac{1-\wt(i, j)}{q-1}
      \frac{q\xchar{\tau}{j} - qt\xchar{\tau}{i}}
      {\xchar{\tau}{j} - qt\xchar{\tau}{i}}\right).
  \]
\end{lem}
\begin{proof}
  From \eqref{eq:c_pi_T}, we have
  \[
    \coeff_{\pi, \wt, \tau} = (-1)^n T_{\sh(\tau)}d_\tau
    \sum_{S \subseteq \corner(\pi)} F_S(\wt) \prod_{j=1}^n \prod_{p=m_j(\pi^S)+1}^{j-1}
    \frac{q\xchar{\tau}{j} - qt\xchar{\tau}{p}}{\xchar{\tau}{j} - qt\xchar{\tau}{p}}.
  \]
  The paths \( \pi^S \) are only different at \( \corner(\pi) \).
  Factoring out the common factors,
\begin{align*}
  &\sum_{S \subseteq \corner(\pi)} F_S(\wt) \prod_{j=1}^n \prod_{p=m_j(\pi^S)+1}^{j-1}
    \frac{q\xchar{\tau}{j} - qt\xchar{\tau}{p}}{\xchar{\tau}{j} - qt\xchar{\tau}{p}} \\
  &= \left(\prod_{j=1}^n \prod_{p=m_j(\pi)+1}^{j-1} \frac{q\xchar{\tau}{j} - qt\xchar{\tau}{p}}{\xchar{\tau}{j} - qt\xchar{\tau}{p}}\right)
    \sum_{S \subseteq \corner(\pi)} F_S(\wt) \prod_{(i, j) \in S}
    \frac{q\xchar{\tau}{j} - qt\xchar{\tau}{i}}{\xchar{\tau}{j} - qt\xchar{\tau}{i}} \\
  &= \left(\prod_{j=1}^n \prod_{p=m_j(\pi)+1}^{j-1} \frac{q\xchar{\tau}{j} - qt\xchar{\tau}{p}}{\xchar{\tau}{j} - qt\xchar{\tau}{p}}\right)
    \prod_{(i, j) \in \corner(\pi)} \left(
    \frac{q-\wt(i, j)}{q-1} - \frac{1-\wt(i, j)}{q-1}
    \frac{q\xchar{\tau}{j} - qt\xchar{\tau}{i}}
    {\xchar{\tau}{j} - qt\xchar{\tau}{i}}\right).
\end{align*}
\end{proof}

\begin{corollary} \label{cor:candidate}
  The coefficient \( \coeff_{\pi, \wt, \tau} = 0 \) if and only if
  \begin{enumerate}
  \item \( \xchar{\tau}{j} = t\xchar{\tau}{p} \) for some \( j,p \) satisfying
    \( 1 \le j \le n \) and \( m_j(\pi)+1 \le p \le j-1 \); or
  \item \( \wt(i, j) = qt\xchar{\tau}{i} / \xchar{\tau}{j}  \) for some
    \( (i, j) \in \corner(\pi) \).
  \end{enumerate}
\end{corollary}
\begin{proof}
  By Lemma~\ref{lem:factorization}, \( \coeff_{\pi, \wt, \tau} = 0 \) if and only if
  \( q\xchar{\tau}{j} - qt\xchar{\tau}{i} = 0 \) for some $(i,j)$ such that \( 1 \le j \le n \) and \( m_j(\pi) + 1 \le i \le j-1 \) or
  \[
    \left( \frac{q-\wt(i, j)}{q-1} - \frac{1-\wt(i, j)}{q-1}
      \frac{q\xchar{\tau}{j} - qt\xchar{\tau}{i}}
      {\xchar{\tau}{j} - qt\xchar{\tau}{i}}\right) = 0
  \]
  for some corner \( (i, j) \in \corner(\pi) \).
  The first equality is equivalent to \( \xchar{\tau}{j} = t\xchar{\tau}{i} \),
  and the second equality is equivalent to \( \wt(i, j) = qt\xchar{\tau}{i} / \xchar{\tau}{j} \).
\end{proof}

\subsection{Proof of the main theorem}
In this subsection, we prove Theorem~\ref{thm:main}.

For brevity, we denote \( \alpha = \alpha(\mu, w) \),
\( \pi = \pi_{\mu, w} \) and \( \wt = \wt_{\mu, w} \).
Assuming \( |w| = \mu_1 \), we construct the tableau \( \tau^* = \tau^*(\mu, w) \)
as follows. For \( 1 \le r \le \ell = \ell(\mu) \), define
\( N_r = \mu_\ell+\mu_{\ell-1}+\cdots+\mu_{\ell-r+1} \), with \( N_0=0 \).
At stage \( r \), order the first \( \mu_{\ell - r + 1} \) columns of \( \mu \)
by decreasing index in \( w \) of their top cells.
In that order, place the integers \( N_{r-1}+1,\ldots,N_r \),
one in each column, in the lowest currently empty cell of that column.
After stage \( \ell \), every cell of \( \dg(\mu) \) has been filled;
the resulting tableau is \( \tau^* \).
See \Cref{fig:rearranging} for an example of \( \tau^* \).

The tableau \( \tau^* \) is indeed standard.
The condition that
\( \mu \setminus \{w_1 ,\dots, w_i\} \in \Par \)
for every \( i \) ensures that, at each stage,
no smaller entry is placed to the right of a larger entry.
Moreover, since we use the smallest unused integers in each stage,
entries are increasing in every column.

The properties below of \( \tau^* \) follow directly from its construction.

\begin{lem}\label{lem:tau_star}
  Let \( \tau^* \) be the standard tableau constructed as above.

  \begin{enumerate}
  \item Let \( u, v \in \dg'(\alpha) \) be the \( i \)th cell and \( j \)th cell in reading order,
    respectively. If \( u \) and \( v \) are vertically adjacent and \( i < j \), then 
    \( \xchar{\tau^*}{j} \) is directly above the cell \( \xchar{\tau^*}{i} \).
    Equivalently, \( \xchar{\tau^*}{j} = t\xchar{\tau^*}{i} \).
  \item Let \( u \in \dg'(\alpha) \) be the \( i \)th cell in reading order.
    Then there are exactly \( l_\alpha(u) \) cells below \( \xchar{\tau^*}{i} \).
    Equivalently, the cell \( \xchar{\tau^*}{i} \) is in the \( (l_\alpha(u)+1) \)th row. In particular, if \( l_\alpha(u) = 0 \), \( \xchar{\tau^*}{i} \) is in the first row.
  \item The set of cells \( \{\xchar{\tau^*}{p} : m_j(\pi)+1 \le p \le j-1\} \)
    forms a maximal horizontal strip of \( \lchar{\tau^*}{j-1} \).
  \end{enumerate}
\end{lem}

We now claim a refinement of \eqref{eq:claim},
\[
  \coeff_{\pi, \wt, \tau} =
  \begin{cases}
    (-1)^{|\mu|}T_{\mu} & \text{if} \quad \tau = \tau^*, \\
    0 & \text{otherwise}
  \end{cases}.
\]

\noindent \textbf{Step 1: Uniqueness of the surviving SYT.}
We first prove that
\( \coeff_{\pi,\wt,\tau}=0 \) whenever \( \tau\neq\tau^* \).
Let \(j\) be the smallest positive integer such that
\( \xchar{\tau}{j} \neq \xchar{\tau^*}{j} \). Note that $\xchar{\tau}{j}$ corresponds to the cell added at the \(j\)th North step of
\( \pi \). We divide the argument into two cases, according to whether
the preceding step is a North step or an East step.

Suppose first that the preceding step is a North step. Then
\( \xchar{\tau^*}{j} \) is the addable cell in the first row of
\( \lchar{\tau^*}{j-1} \). Any other choice of addable cell for
\( \xchar{\tau}{j} \) give
\( \coeff_{\pi,\wt,\tau}=0 \) by part~(1) of
Corollary~\ref{cor:candidate} and part~(3) of Lemma~\ref{lem:tau_star}.

Now suppose that the preceding step is the \(i\)th East step.
If \( \xchar{\tau}{j} \) is not the addable cell in the first row
of \( \lchar{\tau^*}{j-1} \), part~(1) of Corollary~\ref{cor:candidate} gives
\( \coeff_{\pi,\wt,\tau}=0 \).
Now, suppose that \( \xchar{\tau}{j} \) is the addable cell in the first row.
Since \( i < j \), the minimality of \(j\) gives
\(\xchar{\tau^*}{i}=\xchar{\tau}{i}\).
For the corner \( (i,j) \in \corner(\pi) \),
let \( u \) and \( v \) be the \( i \)th cell and the \( j \)th cell
in the reading order of \( \alpha = \alpha(\mu,w)\).
By definition,
\[
  \wt(i,j)=q^{-a_\alpha(u)}t^{l_\alpha(u)+1}.
\]
Note that \(a_\alpha(u)\) counts the cells \(c\in\dg'(\alpha)\) that lie between
\(u\) and \(v\) in the reading order and satisfy
\(l_\alpha(c)\le l_\alpha(u)\). Suppose that \(c\) is the \(k\)th cell in that order. 
Then \(l_\alpha(c)+1\) is the height of \(\xchar{\tau^*}{k}\)
by part~(2) of Lemma~\ref{lem:tau_star}.
Consequently,
the inequality \(l_\alpha(c)\le l_\alpha(u)\) says that the height of \(\xchar{\tau^*}{k}\) is no
larger than the height of \(\xchar{\tau^*}{i}\).
Since \( \lchar{\tau}{j} \) always has partition shape and \( i < k \),
the column containing \(\xchar{\tau^*}{k}\) comes after the column containing \(\xchar{\tau^*}{i}\).
Thus \(a_\alpha(u)\) is exactly
the number of columns to the right of \(\xchar{\tau^*}{i}\) in \( \lchar{\tau^*}{j-1} \).
In addition, by part~(2) of Lemma~\ref{lem:tau_star}, the height of
\( \xchar{\tau^*}{i} \) is \( l_\alpha(u) + 1 \).
Since \(\xchar{\tau^*}{i}=\xchar{\tau}{i}\), while
\(\xchar{\tau}{j}\) is the addable cell in the first row, their relative
positions yield
\[
  \frac{\xchar{\tau}{i}}{\xchar{\tau}{j}} = q^{-a_\alpha(u)-1}t^{l_\alpha(u)},
\]
which is equivalent to part~(2) of Corollary~\ref{cor:candidate}.
  
  \noindent \textbf{Step 2: Computation of the coefficient \( \coeff_{\pi, \wt, \tau^*} \).} \\
  Now we need to show that \( \coeff_{\pi, \wt, \tau^*} = (-1)^nT_{\mu} \).
  It is enough to show that
  \[
    d_{\tau^*} \prod_{j=1}^n\prod_{p=m_j(\pi)+1}^{j-1} \frac{q\xchar{\tau^*}{j}-qt\xchar{\tau^*}{p}}{\xchar{\tau^*}{j}-qt\xchar{\tau^*}{p}}
    \prod_{(i, j) \in \corner(\pi)} \left(
      \frac{q-\wt(i, j)}{q-1} - \frac{1-\wt(i, j)}{q-1}
      \frac{q\xchar{\tau^*}{j} - qt\xchar{\tau^*}{i}}
      {\xchar{\tau^*}{j} - qt\xchar{\tau^*}{i}}\right)
    = 1.
  \]
  By part~(1) of Lemma~\ref{lem:tau_star}, we have
  \( \xchar{\tau^*}{j} = t\xchar{\tau^*}{i} \) for \( (i, j) \in \corner(\pi) \). Hence,   the expression reduces to
  \[
    d_{\tau^*} \prod_{j=1}^n\prod_{p=m_j(\pi)+1}^{j-1} \frac{q\xchar{\tau^*}{j}-qt\xchar{\tau^*}{p}}{\xchar{\tau^*}{j}-qt\xchar{\tau^*}{p}}
    \prod_{(i, j) \in \corner(\pi)} \frac{q-\wt(i, j)}{q-1} 
    = 1.
  \]
  It suffices to prove the following two assertions: if the \( j \)th North
  step is the step after the $(j-1)$th North step, then \begin{equation}\label{eq:individual_edge}
    d_{\lchar{\tau^*}{j-1}, \xchar{\tau^*}{j}}\prod_{p=m_j(\pi)+1}^{j-1}
    \frac{q\xchar{\tau^*}{j}-qt\xchar{\tau^*}{p}}{\xchar{\tau^*}{j}-qt\xchar{\tau^*}{p}} = 1,
  \end{equation}
  and if the \( j \)th North step is the step after the \( i \)th East step, then
\begin{equation}\label{eq:individual_corner}
    d_{\lchar{\tau^*}{j-1}, \xchar{\tau^*}{j}}  \prod_{p=m_j(\pi)+1}^{j-1}
    \frac{q\xchar{\tau^*}{j}-qt\xchar{\tau^*}{p}}{\xchar{\tau^*}{j}-qt\xchar{\tau^*}{p}}
    \left(\frac{q-\wt(i, j)}{q-1}\right) = 1.
  \end{equation}

Recall the definition of \( d_{\lchar{\tau^*}{j-1}, \xchar{\tau^*}{j}} \):
  \[
d_{\lchar{\tau^*}{j-1}, \xchar{\tau^*}{j}} = \prod_{u \in \mathcal{R}_{\lchar{\tau^*}{j-1}, \xchar{\tau^*}{j}}} \frac{q^{a(u)} - t^{l(u) + 1}}{q^{a(u) + 1} - t^{l(u) + 1}}
    \prod_{u \in \mathcal{C}_{\lchar{\tau^*}{j-1}, \xchar{\tau^*}{j}}} \frac{q^{a(u)+1} - t^{l(u)}}{q^{a(u) + 1} - t^{l(u) + 1}},
  \]
where \( \mathcal{R}_{\lchar{\tau^*}{j-1}, \xchar{\tau^*}{j}} \) is the set of cells in \( \lchar{\tau^*}{j-1} \)
  that lie in the same row as \( \xchar{\tau^*}{j} \),
and \( \mathcal{C}_{\lchar{\tau^*}{j-1}, \xchar{\tau^*}{j}} \) is the set of cells in \( \lchar{\tau^*}{j-1} \)
  that lie in the same column as \( \xchar{\tau^*}{j} \).
  Define 
  \[
    \mathcal{W}_{\lchar{\tau^*}{j-1}, \xchar{\tau^*}{j}} = \{ \xchar{\tau^*}{p} : m_j(\pi)+1 \le p \le j-1, \xchar{\tau^*}{p} \text{ lies to the left of } \xchar{\tau^*}{j} \}
  \]
 and
 \[
    \mathcal{E}_{\lchar{\tau^*}{j-1}, \xchar{\tau^*}{j}} = \{ \xchar{\tau^*}{p} : m_j(\pi)+1 \le p \le j-1, \xchar{\tau^*}{p} \text{ lies to the right of } \xchar{\tau^*}{j} \}.
 \]
 See \Cref{fig:RWCE}.

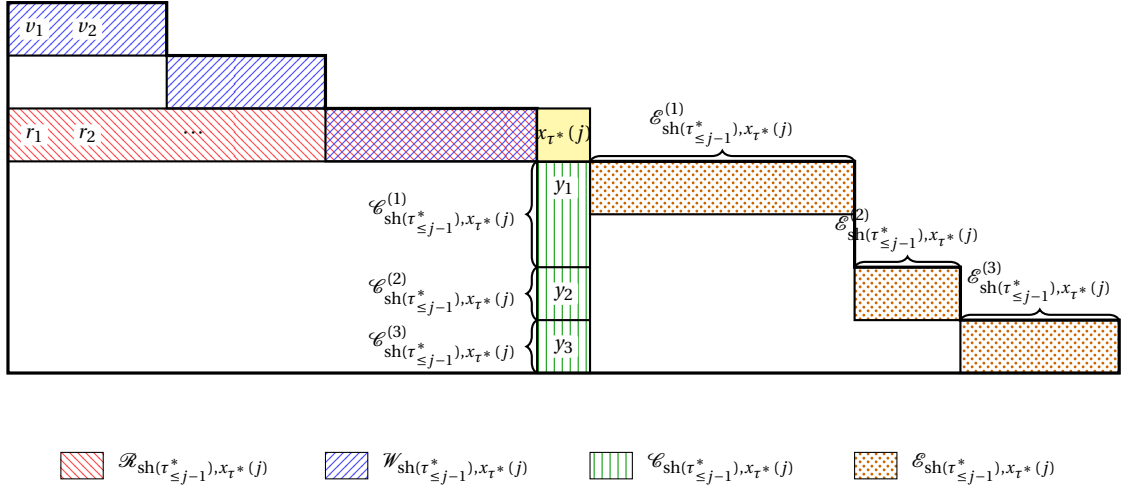
\begin{figure}
  \centering
  \begin{tikzpicture}[scale=0.7, thick]
    
    % Define styles for the regions
    \tikzset{
      Rset/.style={pattern=north west lines, pattern color=red!70},
      Wset/.style={pattern=north east lines, pattern color=blue!70},
      Cset/.style={pattern=vertical lines, pattern color=green!60!black},
      Eset/.style={pattern=crosshatch dots, pattern color=orange!80!black},
      labelbox/.style={fill=white, inner sep=1.5pt, font=\footnotesize, rounded corners=1pt}
    }

    % 1. R(x) Region
    \filldraw[Rset] (0, 4) rectangle (10, 5);

    % 2. W(x) Region 
    % Note: The block (6,4) to (10,5) naturally overlaps with R(x), 
    % the crosshatching of north west (red) and north east (blue) will look great here.
    \filldraw[Wset] (0, 6) rectangle (3, 7);
    \filldraw[Wset] (3, 5) rectangle (6, 6);
    \filldraw[Wset] (6, 4) rectangle (10, 5);

    % 3. C(x) Region
    \filldraw[Cset] (10, 0) rectangle (11, 4);

    % 4. E(x) Region
    \filldraw[Eset] (11, 3) rectangle (16, 4);
    \filldraw[Eset] (16, 1) rectangle (18, 2);
    \filldraw[Eset] (18, 0) rectangle (21, 1);

    % Main Partition Outline \lambda
    \draw[line width=1.2pt] (0, 0) -- (21, 0) -- (21, 1) -- (18, 1) -- (18, 2) -- (16, 2)
    -- (16, 4) -- (10, 4) -- (10, 5) -- (6, 5) -- (6, 6) -- (3, 6) -- (3, 7)
    -- (0, 7) -- cycle;

    % Target Cell x_\tau^*(j)
    \filldraw[fill=yellow!40, draw=black, thick] (10, 4) rectangle (11, 5);
    \node[font=\footnotesize] at (10.5, 4.5) {\( \xchar{\tau^*}{j} \)};

    % Labels for R(x)
    \node[labelbox] at (0.5, 4.5) {\( r_1 \)};
    \node[labelbox] at (1.5, 4.5) {\( r_2 \)};
    \node[fill=none, font=\footnotesize] at (3.5, 4.5) {\( \cdots \)};

    % Labels for W(x)
    \node[labelbox] at (0.5, 6.5) {\( v_1 \)};
    \node[labelbox] at (1.5, 6.5) {\( v_2 \)};

    % Labels for C(x)
    \node[labelbox] at (10.5, 3.5) {\( y_1 \)};
    \draw [decorate, decoration = {brace, amplitude = 4pt}, thick] (10, 2) -- (10, 4)
    node [midway, left, xshift=-4pt, font=\footnotesize]
    {\( \mathcal{C}^{(1)}_{\lchar{\tau^*}{j-1}, \xchar{\tau^*}{j}} \)};
    \node[labelbox] at (10.5, 1.5) {\( y_2 \)};
    \draw (10, 2) -- (11, 2);
    \draw [decorate, decoration = {brace, amplitude = 4pt}, thick] (10, 1) -- (10, 2)
    node [midway, left, xshift=-4pt, font=\footnotesize]
    {\( \mathcal{C}^{(2)}_{\lchar{\tau^*}{j-1}, \xchar{\tau^*}{j}} \)};
    \node[labelbox] at (10.5, 0.5) {\( y_3 \)};
    \draw (10, 1) -- (11, 1);
    \draw [decorate, decoration = {brace, amplitude = 4pt}, thick] (10, 0) -- (10, 1)
    node [midway, left, xshift=-4pt, font=\footnotesize]
    {\( \mathcal{C}^{(3)}_{\lchar{\tau^*}{j-1}, \xchar{\tau^*}{j}} \)};

    % Labels for E(x)
    \draw [decorate, decoration = {brace, amplitude = 4pt}, thick] (11, 4) -- (16, 4)
    node [midway, above, yshift=2pt, font=\footnotesize]
    {\( \mathcal{E}^{(1)}_{\lchar{\tau^*}{j-1}, \xchar{\tau^*}{j}} \)};
    \draw [decorate, decoration = {brace, amplitude = 4pt}, thick] (16, 2) -- (18, 2)
    node [midway, above, yshift=2pt, font=\footnotesize]
    {\( \mathcal{E}^{(2)}_{\lchar{\tau^*}{j-1}, \xchar{\tau^*}{j}} \)};
    \draw [decorate, decoration = {brace, amplitude = 4pt}, thick] (18, 1) -- (21, 1)
    node [midway, above, yshift=2pt, font=\footnotesize]
    {\( \mathcal{E}^{(3)}_{\lchar{\tau^*}{j-1}, \xchar{\tau^*}{j}} \)};
    
    % Legend
    \begin{scope}[shift={(1, -2)}]
      \filldraw[Rset, draw=black, thin] (0, 0) rectangle ++(0.8, 0.5);
      \node[right, font=\footnotesize] at (0.9, 0.2) {\( \mathcal{R}_{\lchar{\tau^*}{j-1},\xchar{\tau^*}{j}} \)};

      \filldraw[Wset, draw=black, thin] (5, 0) rectangle ++(0.8, 0.5);
      \node[right, font=\footnotesize] at (5.9, 0.2) {\( \mathcal{W}_{\lchar{\tau^*}{j-1},\xchar{\tau^*}{j}} \)};

      \filldraw[Cset, draw=black, thin] (10, 0) rectangle ++(0.8, 0.5);
      \node[right, font=\footnotesize] at (10.9, 0.2) {\( \mathcal{C}_{\lchar{\tau^*}{j-1},\xchar{\tau^*}{j}} \)};

      \filldraw[Eset, draw=black, thin] (15, 0) rectangle ++(0.8, 0.5);
      \node[right, font=\footnotesize] at (15.9, 0.2) {\( \mathcal{E}_{\lchar{\tau^*}{j-1},\xchar{\tau^*}{j}} \)};
    \end{scope}

  \end{tikzpicture}
  \caption{The visual description for the sets 
    \( \mathcal{R}_{\lchar{\tau^*}{j-1},\xchar{\tau^*}{j}}, 
    \mathcal{W}_{\lchar{\tau^*}{j-1},\xchar{\tau^*}{j}}, 
    \mathcal{C}_{\lchar{\tau^*}{j-1},\xchar{\tau^*}{j}}, \) and \( 
    \mathcal{E}_{\lchar{\tau^*}{j-1},\xchar{\tau^*}{j}} \) relative to the cell \( \xchar{\tau^*}{j} \).}
  
  \label{fig:RWCE}
\end{figure}
 
 Then \eqref{eq:individual_edge} is equivalent to
 \begin{multline} \label{eq:individual_edge_modified}
   \left( \prod_{r \in \mathcal{R}_{\lchar{\tau^*}{j-1}, \xchar{\tau^*}{j}}}
     \frac{q^{a(r)} - t^{l(r) + 1}}{q^{a(r) + 1} - t^{l(r) + 1}} \right)
   \left( \prod_{z \in \mathcal{C}_{\lchar{\tau^*}{j-1}, \xchar{\tau^*}{j}}}
     \frac{q^{a(z)+1} - t^{l(z)}}{q^{a(z) + 1} - t^{l(z) + 1}}\right) \\
   \left( \prod_{v \in \mathcal{W}_{\lchar{\tau^*}{j-1}, \xchar{\tau^*}{j}}}
     \frac{q\xchar{\tau^*}{j} - qtv}{\xchar{\tau^*}{j} - qtv} \right)
   \left( \prod_{u \in \mathcal{E}_{\lchar{\tau^*}{j-1}, \xchar{\tau^*}{j}}}
     \frac{q\xchar{\tau^*}{j} - qtu}{\xchar{\tau^*}{j} - qtu} \right) = 1
 \end{multline}
 and \eqref{eq:individual_corner} is equivalent to
 \begin{multline}\label{eq:individual_corner_modified}
   \left( \prod_{r \in \mathcal{R}_{\lchar{\tau^*}{j-1}, \xchar{\tau^*}{j}}}
     \frac{q^{a(r)} - t^{l(r) + 1}}{q^{a(r) + 1} - t^{l(r) + 1}} \right)
   \left( \prod_{z \in \mathcal{C}_{\lchar{\tau^*}{j-1}, \xchar{\tau^*}{j}}}
     \frac{q^{a(z)+1} - t^{l(z)}}{q^{a(z) + 1} - t^{l(z) + 1}}\right) \\
   \left( \prod_{v \in \mathcal{W}_{\lchar{\tau^*}{j-1}, \xchar{\tau^*}{j}}}
     \frac{q\xchar{\tau^*}{j} - qtv}{\xchar{\tau^*}{j} - qtv} \right)
   \left( \prod_{u \in \mathcal{E}_{\lchar{\tau^*}{j-1}, \xchar{\tau^*}{j}}}
     \frac{q\xchar{\tau^*}{j} - qtu}{\xchar{\tau^*}{j} - qtu} \right)
   \frac{q-\wt(i, j)}{q-1}
   = 1.
 \end{multline}

 First, we prove \eqref{eq:individual_edge_modified}.
 By the construction of \( \tau^* \),
 \( \xchar{\tau^*}{j} \) is in the last cell in the first row.
 Therefore, there are no cells to the right of or below \( \xchar{\tau^*}{j} \) in $\tau^{*}_{\le j-1}$,
 that is, \( \mathcal{C}_{\lchar{\tau^*}{j-1}, \xchar{\tau^*}{j}} = \emptyset \)
 and \( \mathcal{E}_{\lchar{\tau^*}{j-1}, \xchar{\tau^*}{j}} = \emptyset \).
 Let \( r_1 ,\dots, r_m \) be the cells in \( \mathcal{R}_{\lchar{\tau^*}{j-1}, \xchar{\tau^*}{j}} \) and
 \( v_1 ,\dots, v_m \) be the cells in \( \mathcal{W}_{\lchar{\tau^*}{j-1}, \xchar{\tau^*}{j}} \)
 ordered from the left to the right, respectively.
 We prove \eqref{eq:individual_edge_modified} by showing that
 for each \( 1 \le p \le m \),
 \begin{equation}\label{eq:show_1_left}
   \frac{q^{a(r_p)} - t^{l(r_p)+1}}{q^{a(r_p)+1} - t^{l(r_p)+1}} \frac{q\xchar{\tau^*}{j} - qtv_p}{\xchar{\tau^*}{j} - qtv_p} = 1.
 \end{equation}
 Here, the arm lengths \( a(r_p) \) and the leg lengths \( l(r_p) \) are evaluated
 in the shape \( \lchar{\tau^*}{j-1} \), before the cell \( \xchar{\tau^*}{j} \) is appended.
  Note that \( v_p \) lies \( a(r_p)+1 \) columns to the left
  and \( l(r_p) \) rows above \( \xchar{\tau^*}{j} \).
  Thus
  \[
    \frac{q\xchar{\tau^*}{j} - qtv_p}{\xchar{\tau^*}{j} - qtv_p}
    = \frac{q - qt\left(v_p/\xchar{\tau^*}{j}\right)}{1 - qt\left(v_p/\xchar{\tau^*}{j}\right)}
    = \frac{q - qt(q^{-a(r_p)-1}t^{l(r_p)})}{1 - qt(q^{-a(r_p)-1}t^{l(r_p)})}
    = \frac{q^{a(r_p)+1} - t^{l(r_p)+1}}{q^{a(r_p)} - t^{l(r_p)+1}},
  \]
  and \eqref{eq:show_1_left} holds; therefore, \eqref{eq:individual_edge} holds.

  Now consider \eqref{eq:individual_corner_modified}.
  By the same argument,
  \[
    \left( \prod_{r \in \mathcal{R}_{\lchar{\tau^*}{j-1}, \xchar{\tau^*}{j}}}
     \frac{q^{a(r)} - t^{l(r) + 1}}{q^{a(r) + 1} - t^{l(r) + 1}} \right)
   \left( \prod_{v \in \mathcal{W}_{\lchar{\tau^*}{j-1}, \xchar{\tau^*}{j}}}
     \frac{q\xchar{\tau^*}{j} - qtv}{\xchar{\tau^*}{j} - qtv} \right) = 1.
  \]
  The set \( \mathcal{E}_{\lchar{\tau^*}{j-1}, \xchar{\tau^*}{j}} \) can be partitioned into
  \( \mathcal{E}^{(1)}_{\lchar{\tau^*}{j-1}, \xchar{\tau^*}{j}} ,\dots, \mathcal{E}^{(s)}_{\lchar{\tau^*}{j-1}, \xchar{\tau^*}{j}} \)
  by collecting cells in the same rows, labeled from top to bottom.
  Let \(\mathsf{r}_p\) denote the row index of the cells in
  \(\mathcal{E}^{(p)}_{\lchar{\tau^*}{j-1}, \xchar{\tau^*}{j}}\), with
  \(\mathsf{r}_1>\cdots>\mathsf{r}_s\), and set \(\mathsf{r}_{s+1}=0\).
  The set \( \mathcal{C}_{\lchar{\tau^*}{j-1}, \xchar{\tau^*}{j}} \) can then be partitioned into
  \( \mathcal{C}^{(1)}_{\lchar{\tau^*}{j-1}, \xchar{\tau^*}{j}} ,\dots, \mathcal{C}^{(s)}_{\lchar{\tau^*}{j-1}, \xchar{\tau^*}{j}} \)
  by setting
  \[
    \mathcal{C}^{(p)}_{\lchar{\tau^*}{j-1}, \xchar{\tau^*}{j}}
    = \{z\in\mathcal{C}_{\lchar{\tau^*}{j-1}, \xchar{\tau^*}{j}}:
    \mathsf{r}_{p+1}<\operatorname{row}(z)\le \mathsf{r}_p\}.
  \]
  Let \( y_p \) be the top cell in \( \mathcal{C}^{(p)}_{\lchar{\tau^*}{j-1}, \xchar{\tau^*}{j}} \).
  Since \( \xchar{\tau^*}{j} \) is the addable cell in the column containing \( y_p \),
  we can write \( y_p t^{l(y_p)+1} = \xchar{\tau^*}{j} \).
  Now we have
  \begin{align*}
    &\prod_{u \in \mathcal{E}^{(p)}_{\lchar{\tau^*}{j-1}, \xchar{\tau^*}{j}}} \frac{q\xchar{\tau^*}{j} - qtu}{\xchar{\tau^*}{j} - qtu}
      \prod_{z \in \mathcal{C}^{(p)}_{\lchar{\tau^*}{j-1}, \xchar{\tau^*}{j}}} \frac{q^{a(z)+1} - t^{l(z)}}{q^{a(z)+1} - t^{l(z)+1}} \\
    &\quad= q^{a(y_p)-a(y_{p-1})} \prod_{u \in \mathcal{E}^{(p)}_{\lchar{\tau^*}{j-1}, \xchar{\tau^*}{j}}} \frac{tu - \xchar{\tau^*}{j}}{qtu - \xchar{\tau^*}{j}}
      \prod_{z \in \mathcal{C}^{(p)}_{\lchar{\tau^*}{j-1}, \xchar{\tau^*}{j}}} \frac{q^{a(z)+1} - t^{l(z)}}{q^{a(z)+1} - t^{l(z)+1}}\\
    &\quad= q^{a(y_p)-a(y_{p-1})} \frac{t(q^{a(y_{p-1})+1}y_p) - \xchar{\tau^*}{j}}{qt(q^{a(y_p)}y_p) - \xchar{\tau^*}{j}}
      \frac{q^{a(y_p)+1} - t^{l(y_p)}}{q^{a(y_p)+1} - t^{l(y_{p+1})}} \\
    &\quad= q^{a(y_p)-a(y_{p-1})} \frac{q^{a(y_{p-1})+1} - t^{l(y_p)}}{q^{a(y_p)+1} - t^{l(y_p)}}
      \frac{q^{a(y_p)+1} - t^{l(y_p)}}{q^{a(y_p)+1} - t^{l(y_{p+1})}},
  \end{align*}
  where we set \( a(y_0) = 0 \) and
  \( l(y_{s+1}) = |\mathcal{C}_{\lchar{\tau^*}{j-1}, \xchar{\tau^*}{j}}| \)
  for the boundary condition.
  In the first equality, we factor out the common factor \( q \) from the first product;
  the second equality follows by telescoping both products along with the fact that
  \( a(z) \) stays constant in \( \mathcal{C}_{\lchar{\tau^*}{j-1}, \xchar{\tau^*}{j}} \);
  and the third equality follows from \( \xchar{\tau^*}{j} = y_pt^{l(y_p) + 1} \).
  
  Therefore we have
  \begin{align*}
    \prod_{u \in \mathcal{E}_{\lchar{\tau^*}{j-1}, \xchar{\tau^*}{j}}} \frac{q\xchar{\tau^*}{j} - qtu}{\xchar{\tau^*}{j} - qtu}
    \prod_{z \in \mathcal{C}_{\lchar{\tau^*}{j-1}, \xchar{\tau^*}{j}}} \frac{q^{a(z)+1} - t^{l(z)}}{q^{a(z)+1} - t^{l(z)+1}}
    &=
      \prod_{p=1}^{s} q^{a(y_p)-a(y_{p-1})}
      \frac{q^{a(y_{p-1})+1} - t^{l(y_p)}}{q^{a(y_p)+1} - t^{l(y_p)}}
      \frac{q^{a(y_p)+1} - t^{l(y_p)}}{q^{a(y_p)+1} - t^{l(y_{p+1})}} \\
    &=
      q^{a(y_{s})} \frac{q-1}{q^{a(y_s)+1} - t^{l(y_{s+1})}}
      =
      \frac{q-1}{q - q^{-a(y_s)}t^{l(y_{s+1})}}.
  \end{align*}

  Let \( L_{j-1} \) denote the length of the first row of
  \( \lchar{\tau^*}{j-1} \). By the argument in the first part of this proof,
  \( \wt(i, j) = qt\xchar{\tau^*}{i}/q^{L_{j-1}} \).
  Note that \( a(y_s) \) counts the length difference of the rows
  and \( l(y_{s+1}) = |\mathcal{C}_{\lchar{\tau^*}{j-1},\xchar{\tau^*}{j}}| \).
  Hence, we have
  \[
    \prod_{u \in \mathcal{E}_{\lchar{\tau^*}{j-1}, \xchar{\tau^*}{j}}}
    \frac{q\xchar{\tau^*}{j} - qtu}{\xchar{\tau^*}{j} - qtu}
    \prod_{z \in \mathcal{C}_{\lchar{\tau^*}{j-1}, \xchar{\tau^*}{j}}}
    \frac{q^{a(z)+1} - t^{l(z)}}{q^{a(z)+1} - t^{l(z)+1}}
    \frac{q-\wt(i, j)}{q-1} = 1,
  \]
  which proves \eqref{eq:individual_corner}.
  Combining these identities, we obtain
  \( \coeff_{\pi, \wt, \tau^*} = (-1)^nT_{\mu} \).

This proves the theorem when \( |w|=\mu_1 \). For a general \(w\), let
\(r=\mu_1-|w|\) and choose a maximal extension \(\bar w\) as in
Section~\ref{subsec:statement_of_main_theorem}. By construction,
\(\pi_{\mu,w}\) is obtained by appending \(r\) East steps to
\(\pi_{\mu,\bar w}\), with the same corner weights. Therefore,
\[
  \chi(\pi_{\mu,w},\wt_{\mu,w})
  =d_-^r\chi(\pi_{\mu,\bar w},\wt_{\mu,\bar w})
  =(-1)^{|\mu|}T_\mu d_-^r I_{\mu,\bar w}
  =(-1)^{|\mu|}T_\mu I_{\mu,w},
\]
where the last equality follows by applying \eqref{eq:I_basis_minus}
repeatedly. This proves the theorem for every \((\mu,w)\in\FPI\).

\bibliographystyle{alpha}

\newcommand{\etalchar}[1]{$^{#1}$}

\end{document}